\documentclass{amsart}
\usepackage{amsmath}
\usepackage{amssymb}
\usepackage[a4paper]{geometry}
\usepackage{algorithm}
\usepackage{algpseudocode}
\newtheorem{thm}{Theorem}
\newtheorem{lem}{Lemma}
\newtheorem{claim}{Claim}
\theoremstyle{definition}
\newtheorem{remark}{Remark}
\newcommand{\rr}{\mathbb{R}}

\begin{document}
\title{Colouring multijoints}
\author{Anthony Carbery and Stef\'an Ingi Valdimarsson}
\address{Anthony Carbery, 
School of Mathematics and Maxwell Institute for Mathematical Sciences, 
University of Edinburgh,
JCMB, 
King's Buildings, 
Mayfield Road, 
Edinburgh, EH9 3JZ, 
Scotland.} 
\email{A.Carbery@ed.ac.uk}

\address{Stef\'an Ingi Valdimarsson,
Science Institute,
University of Iceland,
Dunhagi 3,
107 Reykjavik, 
Iceland.} 
\email{siv@hi.is}
\date{6 September 2013}
\maketitle
\begin{abstract}
Let $\mathbb{F}$ be a field, let $L_1, \dots ,L_d$ be pairwise disjoint collections
of lines 
in $\mathbb{F}^d$, and let $\mathcal{L}=\{L_1,\dots,L_d\}$. We say that a point $x\in\mathbb{F}^d$ 
is a \emph{multijoint of $\mathcal{L}$} if $x$ lies on a line from each of the collections in 
$\mathcal{L}$, and moreover the directions of these lines span $\mathbb{F}^d$. 
We prove that there exists a constant $C_d$ such that if
$\mathcal{L}$ is a generic family of collections of lines in $\mathbb{F}^d$ and $J$ is a set of 
multijoints of $\mathcal{L}$,
then there exists a $d$-colouring $\kappa: J \to \{1,2, \dots , d\}$ 
such that for each $j$, for each $l \in L_j$ we have $ | \{ x \in J \cap l \, : \kappa(x) = j\}| \leq C_d |J|^{1/d}$.
\end{abstract}
\section{Introduction}
\noindent
Let $\mathbb{F}$ be a field, let $L_1, \dots ,L_d$ be pairwise disjoint collections
of lines 
in $\mathbb{F}^d$, and let $\mathcal{L}=\{L_1,\dots,L_d\}$. We say that a point $x\in\mathbb{F}^d$ 
is a \emph{multijoint of $\mathcal{L}$} if $x$ lies on a line from each of the collections in 
$\mathcal{L}$, and moreover the directions of these lines span $\mathbb{F}^d$. 

\medskip
\noindent
Regard each line of $L_j$ as being coloured with colour $j$. In this note we address the problem of colouring the set $J$ of 
multijoints of $\mathcal{L}$ with as few colours as possible in such a way that no line of a given colour contains 
too many points of that same colour. We need to make these notions precise, and do so in the statement of our main 
result. Further clarification and a discussion of the context of the result follows in the remarks after its statement. 
The family $\mathcal{L}$ is said to be {\em generic} if whenever $l_j \in L_j$ 
meet at $x$, then the directions of the $l_j$ span $\mathbb{F}^d$.

\begin{thm}
    \label{thmmain}
Let $\mathcal{L}$ be a generic family of collections of lines in $\mathbb{F}^d$ as above, and let $J$ be a (finite) set of 
multijoints of $\mathcal{L}$.
Then there exists a constant $C_d$ which depends only on the dimension
$d$ and not on $\mathcal{L}$ or $J$, and a $d$-colouring $\kappa: J \to \{1,2, \dots , d\}$ 
such that for each $j$, for each $l \in L_j$,
$$ | \{ x \in J \cap l \, : \kappa(x) = j\}| \leq C_d |J|^{1/d}.$$
\end{thm}

\begin{remark} Matters are trivial if we allow more than $d$ colours for $J$: simply colour every point of 
$J$ with colour $(d+1)$. We therefore consider $d$-colourings $ \kappa : J \to \{1,2, \dots , d\}$ of $J$. 
\end{remark}

\begin{remark}
We cannot hope for each line of a 
given colour to contain at most about $|J|^{\beta}$ points of $J$ of the same colour unless $\beta \geq 1/d$. 
To see this consider the monkey-bar/jungle-gym example where $L_j$ consists of $N^{d-1}$ lines parallel to the 
$x_j$-axis passing through the points $(m_1, \dots, m_{j-1}, 0, m_{j+1}, \dots m_{d})$ for $m_i \in \{1, \dots, N \}.$ 
If each line of $L_j$ contains at most $K$ multijoints of colour $j$ then there are at most $N^{d-1}K$
multijoints of colour $j$ altogether and hence at most $dN^{d-1}K$ multijoints altogether. But there are $N^d$
multijoints in this example, so we must have $N^d \leq d N^{d-1}K$. Hence $K$ must satisfy $K \geq N/d = |J|^{1/d}/d$.
\end{remark}

\begin{remark}
We cannot expect in general to use fewer than $d$ colours. To illustrate this in the case $d=3$, put $2N$ red lines 
parallel to $e_1$ passing through the points $(0,j,0)$ for $1 \leq j \leq N$ and $(0,0,j)$
for $1 \leq j \leq N$, and similarly put $2N$ blue lines parallel to $e_2$ and $2N$ green lines parallel 
to $e_3$ in the corresponding places. Then on the plane $x_3 = 0$ we have an $N\times N$ square 
lattice of $N$ red lines parallel to $e_1$ and $N$ blue lines parallel to $e_2$. Through each lattice 
point on this plane put a green line to make it a multijoint in $J$ but in such a way that no new multijoints in $J$ are created.
Similarly add red lines through lattice points on the plane $x_2 = 0$ and blue lines through lattice points on 
the plane $x_1 = 0$. Altogether we now have $3N^2$ multijoints, with the colours red, blue and green in 
symmetry. Can we colour this arrangement of multijoints using only two colours, say red and blue, in such a way that a 
line of a given colour contains at most $\sim N^{2/3}$ points of that colour?
If so, considering the multijoints in the 
the plane $x_3 = 0$, every red line would have at most $\sim N^{2/3}$ red multijoints, so there would be at most
$\sim N \times  N^{2/3} = N^{5/3}$ red multijoints on this plane, and simlarly at most $\sim N^{5/3}$ blue 
multijoints. Hence there would be at most $\sim N^{5/3}$ multijoints on this plane, when in fact there are $\sim N^2$. 
This contradiction shows that we cannot colour this arrangement with fewer than $3$ colours, and similar examples in 
higher dimensions show that in $\mathbb{F}^d$ we will need $d$ colours in general.
\end{remark}

\bigskip
\noindent
Our setting with $d$ families of lines is a variant of the setting of the so-called joints problem.
There we have a single collection $L$ of lines 
in $\mathbb{F}^d$, and we define a \emph{joint of $L$} to be any point which lies at the intersection of $d$ 
lines from $L$ with the condition that the set of directions of those $d$ lines should span $\mathbb{F}^d$.
In recent years there has been quite a bit of interest in the joints problem and it is now known that
if $J$ is the set of
joints of $L$ then we have
\begin{equation}
    |J|\leq C_d |L|^{d/(d-1)}
\end{equation}
where $C_d$ depends only on the dimension $d$.
This was originally proved in the case $\mathbb{F} = \rr$ by Guth and Katz in \cite{MR2680185} for $d=3$, then
for a general $d\geq 3$ by Quilodr\'an \cite{MR2594983} and independently by
Kaplan, Sharir and Shustin \cite{MR2728035}. The extension to general fields is in \cite{CI}. 

\medskip
\noindent
The natural question in 
our setting is whether we have, with $J$ now being the set of multijoints of $\mathcal{L}$,
\begin{equation}\label{MK}
    |J|\leq C_d  \prod_{j=1}^d |L_j|^{1/(d-1)}
\end{equation}
where $C_d$ depends only on the dimension $d$. At the moment this question seems out of reach
except in two dimensions, and we instead consider the 
related problem described above.

\medskip
\noindent
Let us explain the relevance of our result to \eqref{MK}. In \cite{MR2525780}
Dvir proved the finite field Kakeya conjecture. Since then, his central idea, dubbed the
\emph{polynomial method}, has been used extensively, among other things in the
cited work on the joints problem. In another direction, Guth \cite{MR2746348}
extended the polynomial method to prove the endpoint case of the multilinear
Kakeya conjecture in $\rr^d$. This is a continuous version of inequality \eqref{MK} introduced above. 
His proof used algebraic topology but see
\cite{MR3019726} for a treatment which relies only on the Borsuk--Ulam theorem.

\medskip
\noindent
Suppose we have $d$ families $\mathcal{T}_j$ of doubly-infinite tubes $T_j$ in $\rr^d$ of infinite length and 
unit cross-section, and suppose that each tube in $\mathcal{T}_j$ points approximately in the direction
of the $j$-th standard basis vector $e_j$.\footnote{Guth's set up is more relaxed than this, see also \cite{MR2860188}.}
Guth's argument involves a preliminary manipulation and then the main work
goes into proving that for every non-negative function $M$ there exist
functions $S_j$, $j=1,\dots,d$, such that
\begin{align}
    \label{guthone}
    M(x) &\leq \left( \prod_{j=1}^d S_j(x) \right)^{1/d}\quad\text{and}\\
    \label{guthtwo}
    \sum_x S_j(x) &\leq  C_d \|M\|_d.
\end{align}
The domain of the functions $M$ and $S_j$ is the set of unit cubes in $\rr^d$, and inequality
\eqref{guthone} is supposed
to hold for each cube. The sum in inequality \eqref{guthtwo} is over cubes $x$ meeting a tube $T_j \in \mathcal{T}_j$,
and inequality \eqref{guthtwo} is supposed to hold
for each tube in the collection $\mathcal{T}_j$, for all $j=1,\dots,d$.

\medskip
\noindent
In considering inequality \eqref{MK} one is naturally led to consider inequalities \eqref{guthone} 
and \eqref{guthtwo} where the tubes are replaced by lines (of zero width),
the unit cubes by points
and where we suppose that if $x\in l_j\in L_j$ for $j=1,\dots,d$ then the directions of
the $l_j$ should span $\mathbb{F}^d$, i.e. that 
$x$ is a multijoint according to our definition above.
Note that in the case of $\mathbb{R}^d$ straightforward limiting arguments applied to
the results of \cite{MR2746348} and \cite{MR2860188} do not yield an answer
to the question of the satisfiability of \eqref{guthone} and \eqref{guthtwo}, or of
the validity of \eqref{MK} in this setting.

\medskip
\noindent
If in \eqref{guthone} and \eqref{guthtwo} we replace the general nonnegative function $M$ by a 
characteristic function $\chi_J$ of a set of multijoints $J$, and the geometric
mean by the (larger) arithmetic mean we arrive at the (easier) problem of finding $S_j$ such that 

\begin{align}
    \label{cvone}
    \chi_J(x) &\leq \frac{1}{d}\sum_{j=1}^d S_j(x)\quad\text{and}\\
    \label{svtwo}
    \sum_x  S_j(x) &\leq  C_d |J|^{1/d}.
\end{align}

\noindent
Theorem~\ref{thmmain} is equivalent to this new problem: if we have such $S_j$, for each $x$, 
choose a $j$ with $S_j(x) \geq 1$ and assign colour $j$ to $x$; conversely, if we have a colouring 
satisfying the conclusion of Theorem~\ref{thmmain}, declare $S_j(x) = d$ if $x$ has colour $j$ and 
$S_j(x) = 0$ otherwise. 

\medskip
\noindent
Finally, we remark that when $d=2$ there is a simple {\em ad hoc} argument leading to the conclusion of 
Theorem~\ref{thmmain}. Indeed, suppose in $\mathbb{F}^2$ we have a family of blue lines and a family of 
red lines, (with no line having both colours). If a blue line contains at most $\sqrt{2}|J|^{1/2}$ 
bijoints, colour all of those bijoints blue. Colour all other bijoints red. Suppose we have a red line 
with as many as $k = \sqrt{2}|J|^{1/2} + 1$ red bijoints on it. Then each of these bijoints is on a 
(different) blue line, which must therefore contain more than $\sqrt{2}|J|^{1/2}$ bijoints (as they are 
not all blue). Hence there are more than 
$$\sqrt{2}|J|^{1/2} + (\sqrt{2}|J|^{1/2} - 1) + (\sqrt{2}|J|^{1/2} - 2) + \cdots +(\sqrt{2}|J|^{1/2} - (k-1))
= \sqrt{2} k |J|^{1/2} - (k-1)k/2$$  
$$ = \sqrt{2} (\sqrt{2}|J|^{1/2} + 1) |J|^{1/2} - \sqrt{2}|J|^{1/2}(\sqrt{2}|J|^{1/2} + 1) /2 
= |J| + \frac{|J|^{1/2}}{\sqrt{2}} > |J|$$
distinct bijoints altogether, which is a contradiction. Hence each red line also contains at most $\sqrt{2}|J|^{1/2}$
red points too.

\section{Proof of Theorem~\ref{thmmain}}
\noindent
We let $m$ be a positive integer and $J$ be a set of multijoints of $\mathcal{L}$.
We say that a colouring $\kappa:J\to\{1,\dots,d\}$
is \emph{$(m+1)$-unsaturated}
if for any $j=1,\dots,d$ and any
$l_j\in L_j$ we have
$$|\{x\in l_j \, :  \kappa(x)=j\}|\leq m.$$
Otherwise the colouring will be called
\emph{$(m+1)$-saturated}. So in an $(m+1)$-saturated colouring there is some line containing $(m+1)$ 
members of $J$ with the same colour as the line. Theorem~\ref{thmmain} can be restated as:

\begin{thm}
    \label{thmmain2}
Let $\mathcal{L}$ be a generic family of collections of lines in $\mathbb{F}^d$ as above and let $J$ be a set 
of multijoints of $\mathcal{L}$.
Then there exists a constant $C_d$ which depends only on the dimension
$d$ and not on $\mathcal{L}$ or $J$
and an integer $m$ with $m\leq C_d|J|^{1/d}$
such that $J$ is colourable with
an $(m+1)$-unsaturated colouring.
\end{thm}

\begin{proof}
Fix a positive integer $m$ and the set $\mathcal{L}$.
Let $J_c$
be a set of multijoints of $\mathcal{L}$ which is colourable with
an $(m+1)$-unsaturated colouring and let $x_0$ be a
multijoint of $\mathcal{L}$ which does not belong to $J_c$.

\medskip
\noindent
Our aim is to prove the following claim.
\begin{claim}
\label{claim1}
If $m > C_d|J_c|^{1/d}$ then
$\tilde{J}=J_c\cup\{x_0\}$ is colourable with an $(m+1)$-unsaturated colouring.
\end{claim}

\noindent
This claim immediately proves the theorem: every singleton subset of $J$ is trivially colourable with an 
$(m+1)$-unsaturated colouring and the claim allows us to add points one at a time, preserving the property
of being colourable with an $(m+1)$-unsaturated colouring, until the size of the set reaches $(m/C_d)^d$, 
which by assumption will not happen before we exhaust $J$.\footnote{As we shall see below in Section~\ref{trichotomy}, 
our approach constructs a suitable colouring of $\tilde{J}$.}

\medskip
\noindent
We now turn to the proof of the claim.
To simplify notation we use $J$ for what is called $J_c$ in the statement of the claim.
Let us label the elements of $J=\{x_1,\dots,x_\nu\}$ and order the multijoints
in $\tilde{J}$ according to the indices.
Let $K$ be the set of colourings of $J$ which are $(m+1)$-unsaturated.

\medskip
\noindent
We wish to define a strict partial ordering on $K$. To do this we construct
for every $\kappa\in K$ a {\bf coloured rooted tree} $T$ whose vertices belong to
$\tilde{J}$.
The tree will be rooted at $x_0$ (which is achromatic) and all the other vertices will be members of
$J$ and will be assigned the colour given to them by $\kappa$. (We shall not colour the edges of the tree.) 

\subsection{Construction of the tree}\label{tree}
Let us fix a $\kappa\in K$ and describe the construction of the tree $T$
with an iterative process.

\medskip
\noindent
At the $0$-th step, the tree $T_0$ has one vertex, $x_0$, and no
edges. We will maintain an ordering on the vertices, based primarily
on the step in which a vertex gets added and secondarily on
the ordering inherited from $\tilde{J}$. In accordance with that we say that
$x_0$ is the first element of the tree and give it the alternative name $y_1$.

\medskip
\noindent
At the $i$-th step we consider $T_{i-1}$, and either construct a $T_i$, or else stop the procedure and declare
$T := T_{i-1}$. We consider the $i$-th element of the tree $T_{i-1}$, which we call $y_i$, and construct 
$T_i$ by adding one or more children from amongst the members of $\tilde{J}$ not already in $T_{i-1}$ to $T_i$, 
and then connect $y_i$ to its children with edges. 

\medskip
\noindent
If there is no $i$-th element in the tree $T_{i-1}$ we say that the tree is \emph{fully constructed}
and we define $T := T_{i-1}$. Clearly this must happen before or when we reach step $|J|$.

\medskip
\noindent
Otherwise, we proceed as follows. The children of $y_i$ will be the elements of $\tilde{J}$ which are not
already in the tree, and which are the reasons that we may not change the colour of $y_i$ without the colouring 
becoming $(m+1)$-saturated or at least in danger of becoming so.\footnote{The precise significance of this will become clearer 
as the proof proceeds.} Specifically, for each colour $j$ different from $\kappa(y_i)$\footnote{In the case $i=1$ 
this simply means all colours $j$. This understanding applies in several places below.}
let $L_j^{(i)}$ be the subset of $L_j$ consisting of the lines $l_j$ going
through $y_i$ such that 
\begin{equation}
\label{eqcond}
|\{x\in l_j\cap J\, :  \kappa(x)=j\}| + 
|\{x\in l_j\cap T_{i-1}\cap J\, :  \kappa(x)\neq j\}| \geq m.
\end{equation}
(Here and later we abuse notation and use $T_{i-1}$ also to denote the vertex set of the tree
$T_{i-1}$.) If the collection $L_j^{(i)}$ is 
empty for some $j\neq \kappa(y_i)$ we say that the colouring is \emph{advanceable at step $i$}.\footnote{See
Section~\ref{trichotomy} Claim~\ref{claim2}(b) for the reason we use this 
terminology.} Then we stop the construction, and declare $T := T_{i-1}$.

\bigskip
\noindent
Otherwise, if the colouring is not advanceable at step $i$, we define for each
$j\neq\kappa(y_i)$
the set $I_j^{(i)}$ which consists of all the points of $J$
of colour $j$ on any line
from $L_j^{(i)}$, excluding those points which are already in the tree
$T_{i-1}$. We let the tree $T_i$ be the tree $T_{i-1}$ with the points from
$I_j^{(i)}$ for all $j\neq \kappa(y_i)$ added as vertices, specifically
as children of $y_i$. The edges of $T_i$ are those of $T_{i-1}$ together
with edges linking $y_i$ to each of its children. Note that no child has the same colour as its parent.

\bigskip
\noindent
We remark that it is possible that $I_j^{(i)}$ may be empty for some colour
$j$ even though $L_j^{(i)}$ is non-empty. This is the case if all
the points of colour $j$ in $J$ which lie on any line in $L_j^{(i)}$
are vertices of $T_{i-1}$. If all of the $I_j^{(i)}$ are empty we let $T_i := T_{i-1}$.

\medskip
\noindent
If the colouring is not advanceable at any step then eventually the tree
will become fully constructed. We call such a colouring \emph{non-advanceable}.

\subsection{The first stage in the construction}\label{firststage}
To fix ideas, let us run through the first stage of the construction. We have $T_0 = \{y_1\} = \{x_0\}$
and so there is a first element of $T_0$. For each $j \in \{1,2, \dots, d\}$ we have that 
$L_j^{(1)}$ is the subset of $L_j$ consisting of the lines $l_j$ going
through $y_1$ such that 
\begin{equation}
\label{eqcondd}
|\{x\in l_j\cap J \, :   \kappa(x)=j\}| \geq m
\end{equation}
since the second term on the left-hand side of \eqref{eqcond} is zero. Since $\kappa$ is $(m+1)$-unsaturated, 
\eqref{eqcondd} means that $|\{x\in l_j\cap J \, : \kappa(x)=j\}|= m$. Now either $L_j^{(1)}$ is empty for some 
colour $j$, or it is non-empty for all $j$.
\begin{enumerate}
\item In the first case we have that $\kappa$ is advanceable at step $1$, 
we stop the procedure and declare $T = T_0$. {\em Note that in this case, if for a certain $j$, $L_j^{(1)} = \emptyset$, 
then for all $l_j \in L_j$ passing through $y_1$ we have 
$$|\{x\in l_j\cap J \, :  \kappa(x)=j\}| < m.$$
In this case we can simply assign the colour $j$ to $x_0$ and we are done.}\footnote{This observation will be 
important below.}
\item Otherwise, when $L_j^{(1)} \neq \emptyset$ for all $j$, we have 
$$I_j^{(1)} = \{x \in J \, :  \kappa(x) = j \mbox{  and  } x \in l_j \mbox{  for some } l_j \in L_j^{(1)} \}$$
and we note that since every line in each $L_j^{(1)}$ has exactly $m$ members of $J$ of colour $j$ on it, each $ 
I_j^{(1)}$ has at least $m$ members, and so $T_1$ will be a proper extension of $T_0$, and in particular will have a 
second member ready for the construction of $T_2$.
\end{enumerate}

\medskip
\noindent
The construction of the tree $T$ depends on the colouring $\kappa$; when we wish to emphasise this we shall use 
the notation $T(\kappa)$ and likewise $I_j^{(i)}(\kappa)$ to highlight this dependence.

\subsection{A strict partial ordering}
Now we turn to the definition of the strict partial ordering on $K$.
Take $\kappa_1,\kappa_2\in K$ and construct the trees
$T(\kappa_1)$ and $T(\kappa_2)$.
We say that \emph{$\kappa_1$ is more advanced
than $\kappa_2$ at level $i_0$} if
\renewcommand{\theenumi}{\roman{enumi}}
\begin{enumerate}
\item $I_j^{(i)}(\kappa_1) =I_j^{(i)}(\kappa_2)$
for all $j\neq \kappa_1(y_i)$ and for all $i<i_0$;
\item
$I_j^{(i_0)}(\kappa_1)\subseteq I_j^{(i_0)}(\kappa_2)$ for all
$j\neq \kappa_1(y_{i_0})$;
\item at least one of the inclusions in item 
(ii) is strict.
\end{enumerate}

\noindent
In other words, $\kappa_1$ is more advanced than $\kappa_2$ at level $i_0$ if the 
coloured trees $T_{i_0 - 1}(\kappa_1)$ and $T_{i_0 - 1}(\kappa_2)$ are identical,
and $T_{i_0}(\kappa_1)$ is a proper coloured subtree of $T_{i_0}(\kappa_2)$.\footnote{Note that 
this expresses the idea that the construction of the tree for $\kappa_1$ as in the previous subsection 
is closer to termination than that for  $\kappa_2$; hence the terminology ``more advanced''.}

\medskip
\noindent
We say that $\kappa_1$ is \emph{more advanced} than $\kappa_2$
if there is a level $i_0$
such that $\kappa_1$ is more advanced than $\kappa_2$ at level $i_0$.
Note that there can be at most one such level because of the requirement
of a strict inclusion at level $i_0$. It is clear that the notion of being 
more advanced is a strict partial ordering on $K$.

\subsection{A trichotomy}\label{trichotomy}
Now, for a general colouring $\kappa\in K$ there are three possibilities.
It may be advanceable at step $1$, it may be advanceable at some step
$i>1$ or it may be non-advanceable. We will prove the following claim.
\begin{claim}
\label{claim2}
\mbox{}
\renewcommand{\theenumi}{\alph{enumi}}
\begin{enumerate}
\item If $\kappa$ is advanceable at step $1$ then we can
extend $\kappa$ to an $(m+1)$-unsaturated colouring of $\tilde{J}$.
\item If $\kappa$ is advanceable at some step $i>1$ then
there is a colouring $\tilde\kappa \in K$ which
is more advanced than $\kappa$.
\item If $\kappa$ is non-advanceable then
$|T(\kappa)\cap J|\geq C_d^{-d} m^d$.
\end{enumerate}
\end{claim}

\noindent
We will establish Claim~\ref{claim2} below, but for now we note that 
Claim~\ref{claim1} follows immediately from it. Indeed,
the hypothesis of Claim~\ref{claim1} is that $m > C_d|J|^{1/d}$, 
so
$$|T(\kappa)\cap J|\leq |J|< C_d^{-d} m^d,$$
meaning that under the hypothesis of Claim~\ref{claim1} the third alternative cannot hold for any $\kappa \in K$.
So every $\kappa \in K$ must be advanceable at some step. If $\kappa$ is advanceable at step $1$, part (a) 
of Claim~\ref{claim2} gives us what we want; if not, $\kappa$ will be advanceable at some step $i > 1$ and there will be a 
$\tilde{\kappa} \in K$ which is more advanced than $\kappa$. Once again, the third alternative cannot hold for 
$\tilde{\kappa}$, if the first alternative holds we are happy, and if the second alternative holds we obtain a 
$\tilde{\tilde{\kappa}}$ which is more advanced than $\tilde{\kappa}$. We iterate this process. Since $K$ is finite, 
a maximally advanced element of $K$ must exist, meaning that at some point of the iteration the second alternative 
cannot hold, leaving us with only the first. In summary, if $m > C_d|J|^{1/d}$, for every $\kappa \in K$ there is some 
$\tilde{\kappa} \in K$ which is more advanced than $\tilde{\kappa}$ and which is advanceable at step $1$. Hence there 
exists an $(m+1)$-unsaturated colouring of $\tilde{J}$ as required. 

\medskip
\noindent
This procedure gives an algorithm for actually constructing an $(m+1)$-unsaturated
colouring, see Algorithm~\ref{alg1}.

\begin{algorithm}
\caption{Construct an $(m+1)$-unsaturated colouring of a set $J$ of multijoints}
\label{alg1}
\begin{enumerate}
\renewcommand{\theenumi}{\arabic{enumi}}
\setcounter{enumi}{-1}
\item We require $m>C_d|J|^{1/d}$.
\item Let $J_c$ be the empty set and $\kappa$ be a colouring of $J_c$.
\item For each point $x_0$ of $J$ do the following:
    \begin{enumerate}
        \item While $\kappa$ is advanceable at some step $i>1$ w.r.t. $\tilde{J}=J_c\cup\{x_0\}$ do the following:
            \begin{enumerate}
                \item Let $\tilde\kappa$ be a colouring of $J_c$ which is more advanced than $\kappa$,
                    constructed as in the proof of Claim~\ref{claim2} (b).
                \item Update $\kappa$ to be $\tilde\kappa$.
            \end{enumerate}
        \item Now $\kappa$ is advanceable at step 1.
        \item Extend $\kappa$ to $\tilde{J}$ by letting $\kappa(x_0)$ be some colour $j$ for which $L_j^{(1)}$ is empty.
        \item Update $J_c$ to be $\tilde{J}$.
    \end{enumerate}
\item Now $\kappa$ is an $(m+1)$-unsaturated colouring of $J$.
\end{enumerate}
\renewcommand{\theenumi}{\alph{enumi}}
\end{algorithm}

\medskip
\noindent
Note that we have already established case (a) in the discussion
of case (1) in Section~\ref{firststage}.
In the remainder of the proof we will verify the remaining two cases of
Claim~\ref{claim2}.

\subsection{Establishing Claim~\ref{claim2}(b)}
For the second case of Claim~\ref{claim2},
let us assume that $\kappa$ is advanceable at step $i_0>1$.
That means that there is a colour $j_0\neq\kappa(y_{i_0})$
such that for all lines
$l_{j_0}\in L_{j_0}$ such that $y_{i_0}\in l_{j_0}$ we have
\begin{equation}
\label{eq2}
|\{x\in l_{j_0}\cap J \; : \kappa(x)=j_0\}| +
|\{x\in l_{j_0}\cap T_{i_0-1}\cap J \, : \kappa(x)\neq j_0\}| < m.
\end{equation}
We let $j_1=\kappa(y_{i_0})$.
Let us define a new colouring $\tilde{\kappa}$ which is identical to $\kappa$
except that $\tilde\kappa(y_{i_0})=j_0$. We need to consider the effect of changing the colour of
$y_{i_0}$ on the construction of the tree $T(\tilde{\kappa})$, and in particular, we need to bear in mind 
the possibility that $y_{i_0}$ might occur earlier in the construction of $T(\tilde{\kappa})$ than of $T(\kappa)$.
As a rough guide, note that lines of colours other than the old and new colours of $y_{i_0}$ will play exactly the same role 
in both constructions, as will lines not containing $y_{i_0}$. We will need to examine vertices of the tree 
$T(\kappa)$ of colour either the old or new colour of $y_{i_0}$ for possible changes in the construction.

\medskip
\noindent
Specifically, we wish to verify that $\tilde\kappa$ is $(m+1)$-unsaturated and that
$\tilde\kappa$ is more advanced than $\kappa$.
The $(m+1)$-unsaturated conditions for $\tilde\kappa$ follow immediately
from the corresponding conditions for $\kappa$ except for lines of
colour $j_0$ which go through $y_{i_0}$. But for those lines we just saw
that
$$
|\{x\in l_{j_0}\cap J \, : \kappa(x)=j_0\}| < m.
$$
and so
$$
|\{x\in l_{j_0}\cap J \, : \tilde\kappa(x)=j_0\}|=
|\{x\in l_{j_0}\cap J \, : \kappa(x)=j_0\}|+1 \leq m.
$$
Thus $\tilde\kappa$ belongs to $K$ and it is meaningful to ask whether
$\tilde\kappa$ is more advanced than $\kappa$.

\medskip
\noindent
Since $y_{i_0}\in T(\kappa)$ we can find an index $i_1<i_0$ such that
$y_{i_0}$ is a vertex of $T_{i_1}(\kappa)$ but not of $T_{i_1-1}(\kappa)$.
That means that either $i_1=1$ or there exists a colour $j_2\neq j_1$ such that
$\kappa(y_{i_1})=j_2$ and a line $l_{j_1}\in L_{j_1}$ such that
$y_{i_0},y_{i_1}\in l_{j_1}$ and
\begin{equation}
\label{eq3}
|\{x\in l_{j_1}\cap J \, :\kappa(x)=j_1\}| +
|\{x\in l_{j_1}\cap T_{i_1-1}\cap J \, : \kappa(x)\neq j_1\}| \geq m.
\end{equation}
{\bf We want to show that $\tilde\kappa$ is more advanced than $\kappa$ at
level $i_1$.}
First let us verify 
condition (i).
If $i_0=1$ then condition (i) is vacuous.
Otherwise it is clear that there are two
types of steps we have to consider, and for other steps before $i_1$
condition (i) is immediate. The types of steps we have to consider correspond to 
vertices of the tree $T(\kappa)$ of colour equal to the new colour of $y_{i_0}$ and 
of colour equal to the old colour of $y_{i_0}$,
and more precisely are:
\begin{itemize}
\item steps $i_2$ such that $i_2<i_1$ and $\kappa(y_{i_2})\neq j_0$
but there is a line $l_{j_0}\in L_{j_0}$ such that $y_{i_2},y_{i_0}\in l_{j_0}$;
and
\item steps $i_3$ such that $i_3<i_1$ and $\kappa(y_{i_3})\neq j_1$
but there is a line
$\tilde{l}_{j_1}\in L_{j_1}$ such that $y_{i_3},y_{i_0}\in \tilde{l}_{j_1}$.
\end{itemize}
For the former case we note that $y_{i_0}\not\in T_{i_2}(\kappa)$ since
this vertex set is a subset of $T_{i_1-1}(\kappa)$ which by assumption
$y_{i_0}$ does not belong to.
Therefore we see that as $i_0 > i_2$
$$
|\{x\in l_{j_0}\cap T_{i_0-1}\cap J \, : \kappa(x)\neq j_0\}|\geq 
|\{x\in l_{j_0}\cap T_{i_2-1}\cap J \, : \kappa(x)\neq j_0\}|+1
$$
since $y_{i_0}$ is a member of the former set but not the latter.
So \eqref{eq2} shows that
$$
|\{x\in l_{j_0}\cap J \, : \kappa(x)=j_0\}| +
|\{x\in l_{j_0}\cap T_{i_2-1}\cap J \, : \kappa(x)\neq j_0\}| < m-1.
$$
Now note that
$$
|\{x\in l_{j_0}\cap J \, : \tilde\kappa(x)=j_0\}|=
|\{x\in l_{j_0}\cap J \, : \kappa(x)=j_0\}|+1
$$
and
$$
|\{x\in l_{j_0}\cap T_{i_2-1}\cap J \, : \tilde\kappa(x)\neq j_0\}|
=
|\{x\in l_{j_0}\cap T_{i_2-1}\cap J \, : \kappa(x)\neq j_0\}|
$$
where the second equality follows since $\tilde\kappa$ and $\kappa$ are
identical on $T_{i_2-1}$.
So we obtain
$$
|\{x\in l_{j_0}\cap J \, : \tilde\kappa(x)=j_0\}| +
|\{x\in l_{j_0}\cap T_{i_2-1}\cap J \, : \tilde\kappa(x)\neq j_0\}| < m
$$
and this shows that $l_{j_0}\not\in L_{j_0}^{(i_2)}(\tilde\kappa)$.
Moreover,  $l_{j_0}\not\in L_{j_0}^{(i_2)}(\kappa)$ since if it were in this
set then we would have $y_{i_0}\in T_{i_2}(\kappa)$
which is not possible as we saw above. So $y_{i_0}$ is not added to the vertices of $T(\tilde{\kappa})$ at 
step $i_2$, and we deduce that the iteration in the definition of the trees
proceeds identically at this step for $\kappa$ and $\tilde\kappa$. That is, the coloured trees $T_{i_2}(\tilde{\kappa})$
and $T_{i_2}(\kappa)$ are identical.

\medskip
\noindent
For the latter case we note that $y_{i_0}\not\in T_{i_3}(\kappa)$ since
this vertex set is a subset of $T_{i_1-1}(\kappa)$ which by assumption
$y_{i_0}$ does not belong to.
That means that
$$
|\{x\in \tilde{l}_{j_1}\cap J \, : \kappa(x)=j_1\}| + 
|\{x\in \tilde{l}_{j_1}\cap T_{i_3-1}\cap J \, : \kappa(x)\neq j_1\}| < m.
$$
Now note that
$$
|\{x\in \tilde{l}_{j_1}\cap J \, : \tilde\kappa(x)=j_1\}|=
|\{x\in \tilde{l}_{j_1}\cap J \, : \kappa(x)=j_1\}|-1
$$
and
$$
|\{x\in \tilde{l}_{j_1}\cap T_{i_3-1}\cap J \, : \tilde\kappa(x)\neq j_1\}|
=
|\{x\in \tilde{l}_{j_1}\cap T_{i_3-1}\cap J \, : \kappa(x)\neq j_1\}|
$$
where the second equality follows since $\tilde\kappa$ and $\kappa$ are
identical on $T_{i_3-1}$.
Therefore
$$
|\{x\in \tilde{l}_{j_1}\cap J \, : \tilde\kappa(x)=j_1\}| + 
|\{x\in \tilde{l}_{j_1}\cap T_{i_3-1}\cap J \, : \tilde\kappa(x)\neq j_1\}| < m-1,
$$
and this shows that $\tilde{l}_{j_1}\not\in L_{j_1}^{(i_3)}(\tilde\kappa)$.
Moreover,
$\tilde{l}_{j_1}\not\in L_{j_1}^{(i_3)}(\kappa)$ since if it were in this
set then we would have $y_{i_1}\in T_{i_3}(\kappa)$
which is not possible as we saw above. So $y_{i_0}$ is not added to the vertices of $T(\tilde{\kappa})$ at 
step $i_3$, and we deduce that the iteration in the definition of the trees
proceeds identically at this step for $\kappa$ and $\tilde\kappa$. That is, the coloured trees $T_{i_3}(\tilde{\kappa})$
and $T_{i_3}(\kappa)$ are identical.

\medskip
\noindent
Hence we conclude that the coloured trees $T_{i_1 - 1}(\tilde{\kappa})$
and $T_{i_1 - 1}(\kappa)$ are identical.

\medskip
\noindent
Now we verify conditions 
(ii) and (iii). 
We note that the only possible
difference between the sets $L_j^{(i_1)}(\kappa)$ and $L_j^{(i_1)}(\tilde{\kappa})$ for some colour
$j$ is that a line containing both $y_{i_0}$ and $y_{i_1}$ could be in one of these sets and not the other.
We already know that the line joining these points is of colour $j_1$ so for other colours
we have that $L_j^{(i_1)}(\kappa)$ and $L_j^{(i_1)}(\tilde{\kappa})$ are identical and so
$I_j^{(i_1)}(\kappa)$ and $I_j^{(i_1)}(\tilde{\kappa})$ are identical too.
For colour $j_1$ we have that $\kappa(y_{i_0})=j_1\neq\tilde{\kappa}(y_{i_0})$.
This shows that $y_{i_0}$ is an
element of $I_{j_1}^{(i_1)}(\kappa)$ but not of $I_{j_1}^{(i_1)}(\tilde{\kappa})$.
Hence we conclude that the coloured tree $T_{i_1}(\tilde{\kappa})$
is a proper coloured subtree of $T_{i_1}(\kappa)$
and conditions (ii) and (iii) are verified.

\medskip
\noindent
This establishes the second case of Claim~\ref{claim2}.

\subsection{Establishing Claim~\ref{claim2}(c)}

For the last case of Claim~\ref{claim2},
let us recall the statement of Quilodr\'an's lemma, \cite{MR2594983} and \cite{CI}, see also \cite{MR2728035}.
\begin{lem}
    Let $L$ be a collection of lines in $\mathbb{F}^d$ and 
    let $J$ be a subset of the set of joints of $L$. Suppose that $J$ has the
    property that for every line $l\in L$ the cardinality of $l\cap J$ is
    either $0$ or at least $m$. Then $|J|\geq C_d m^d$.
\end{lem}
\noindent
Let us then assume that $\kappa$ is non-advanceable.
We let $\bar{J}$ be the set of points in $J$ which are vertices of the tree
$T(\kappa)$. For a colour $j$ we let $\bar{L}_j$ be the union
of the sets $L_j^{(i)}$ for those indices $i$
such that $\kappa(y_i)\neq j$.
Then we let $\bar{L}=\bar{L}_1\cup\dots\cup\bar{L}_d$.

\medskip
\noindent
We need to verify that the hypotheses of the lemma are satisfied for
$\bar{J}$ and 
$\bar{L}$.
First we note that the elements of $\bar{J}$
are in fact joints of 
$\bar{L}$.
To see this, 
take $y_i\in\bar{J}$ with $i>1$
and assume that $\kappa(y_i)=j$. Then there is an $\tilde\imath<i$ and
a line $l_j\in L_j^{(\tilde\imath)}$ such that $y_i\in l_j$ and $\kappa(y_{\tilde{\imath}}) \neq j$
(since children always have a different colour from their parents). So $y_i \in \bar{L}_j$.
Furthermore, for all colours $\tilde\jmath\neq j$ we have by non-advanceability that
the set $L_{\tilde\jmath}^{(i)}$ is non-empty and the lines in these
sets all go through $y_i$. So $y_i \in \bar{L}_{\tilde\jmath}$. Thus for each $j_\ast \in \{1, \dots, d\}$, we have 
$y_i \in \bar{L}_{j_\ast} \subseteq L_{j_\ast}$. Since by hypothesis the collection $\mathcal{L}$ is generic, 
we conclude that $y_i$ is a joint of $\bar{L}$.



\bigskip
\noindent
Now consider a line in $\bar{L}$, say $l_j\in L_j^{(i)}$. Then by definition of $L_j^{(i)}$ we have 
$$
|\{x\in l_j\cap J\,:   \kappa(x)=j\}| + 
|\{x\in l_j\cap T_{i-1}\cap J \, :  \kappa(x)\neq j\}| \geq m.
$$
Note that all the points which are elements of the first of these sets
will be vertices of $T_i$. Therefore the two sets occuring in this
expression are subsets of $l_j \cap \bar{J}$ which are disjoint. 
Hence we have $|l_j\cap \bar{J}|\geq m$.
This shows that all the hypotheses of the lemma are satisfied and so
we deduce that
$|\bar{J}|\geq C_d^{-d} m^d$.
\end{proof}
\begin{remark}
The reader will observe that we use the hypothesis of genericity only in establishing Claim 2(c).
We conjecture that the main result holds without this hypothesis.
\end{remark}
\bibliographystyle{plain}
\bibliography{colouring}
\end{document}